\documentclass[reqno,11pt]{amsart} \numberwithin{equation}{section}
\input{amssym.def}
\input{amssym.tex}
\begin{document}
\title[Weak Theorems for Symmetric Generalized Hybrid...] {Weak convergence theorems for symmetric generalized hybrid
mappings in uniformly convex Banach spaces}
\author[F. Moradlou and S. Alizadeh]{Fridoun Moradlou$^1$ and Sattar Alizadeh$^2$}
\address{\indent $^{1,2}$ Department of Mathematics
\newline \indent Sahand University of Technology
\newline \indent Tabriz, Iran}
\email{\rm $^1$ moradlou@sut.ac.ir \& fridoun.moradlou@gmail.com}
\email{\rm $^2$ sa\_alizadeh@sut.ac.ir}
\thanks{}
\begin{abstract}
In this paper, we  prove  some theorems related to properties of
generalized symmetric hybrid mappings in Banach spaces. Using Banach
limits, we prove a fixed point theorem  for  symmetric generalized
hybrid mappings in  Banach spaces. Moreover, we prove some weak
convergence theorems  for such mappings by using Ishikawa iteration
method in a uniformly convex Banach space.
\end{abstract}


\maketitle
\baselineskip=15.8pt \theoremstyle{definition}
  \newtheorem{df}{Definition}[section]
    \newtheorem{rk}[df]{Remark}
\theoremstyle{plain}
 \newtheorem{lm}[df]{Lemma}
  \newtheorem{thm}[df]{Theorem}
  \newtheorem{exa}[df]{example}
  \newtheorem{cor}[df]{Corollary}
  \setcounter{section}{0}

\section{Introduction}
Let $C$ be a nonempty, closed convex subset of a real Banach space
$E$. The self mapping $T$ of $C$ is called \textit{nonexpansive} if
$\|Tx-Ty\|\leq \|x-y\|$ for all $x,y \in C$. A point $x \in C$ is
fixed point of $T$ provided $Tx=x$. We denote by $F(T )$ the set of
fixed points of $T$.
\par
There exist some iteration processes which often used to approximate
a fixed point of a nonexpansive mapping: Picard iteration,
Krasnoselskii iteration, Halpern iteration, Mann iteration and
Ishikawa iteration. During the recent years,  Mann and Ishikawa
iterative schemes \cite{  ishikawa, mann} have been studied by a
number of authors.
\par
Let $E$ be a nonempty closed convex subset of a Banach space. In
1953, for a self mapping $T$ of $E$, Mann \cite{mann} defined the
following iteration procedure:
\begin{equation}\label{mann}
\begin{cases}
 & x_{0} \in C \,\,\,\hbox{chosen arbitrarily},\\
 & x_{n+1}  = \alpha_{n}x_{n}+(1-\alpha_{n})Tx_{n},
\end{cases}
\end{equation}
where $0\leq \alpha_{n}\leq 1$ for all $n \in \mathbb{N}\cup \{0\}$.
\par
 Let $K$ be a closed convex subset of a Hilbert space $H$.
 In 1974, for a Lipschitzian pseudocontractive self mapping $T$ of $K$, Ishikawa \cite{ishikawa}
defined the following iteration procedure:
\begin{equation}\label{ishikawa}
\begin{cases}
  & x_{0} \in C \,\,\,\hbox{chosen arbitrarily},\\
  & y_{n}= \beta_{n}x_{n}+ (1-\beta_{n})Tx_{n}, \\
  & x_{n+1}  = \alpha_{n}x_{n}+(1-\alpha_{n})Ty_{n},
\end{cases}
\end{equation}
where $0\leq \beta_{n} \leq \alpha_{n}\leq 1$ for all $n \in
\mathbb{N}\cup \{0\}$ and he proved strong convergence of the
sequence $\{x_{n}\}$ generated by the above iterative scheme if
$\lim_{n \to \infty} \beta_{n} =1$ and $\sum_{n=1}^{\infty}
(1-\alpha_{n})(1-\beta_{n}) = \infty$. By taking $\beta_{n} =1$ for
all $n \geq 0$ in $(\ref{ishikawa})$, Ishikawa iteration process
reduces to Mann iteration process.
\par
In general, to gain the convergence in Mann and Ishikawa iteration
processes, we must assume that underlying space $E$ has elegant
properties. For example, Reich \cite{Reich} proved that if $E$ is a
uniformly convex Banach space with a Fr\'{e}chet differentiable norm
and if $\{\alpha_{n}\}$ is such that $\sum_{n=1}^{\infty} \alpha_{n}
(1- \alpha_{n}) = \infty$, then the Mann iteration scheme converges
weakly to a fixed point of $T$. However, we know that the Mann
iteration process is weakly convergent even in a Hilbert space
\cite{Genel}. Also, Tan and Xu \cite{Tan} proved that if $E$ is a  a
uniformly convex Banach space which satisfies Opial's condition or
whose norm is Fr\'{e}chet differentiable and if $\{\alpha_{n}\}$ and
$\{\beta_{n}\}$ are such that $\sum_{n=1}^{\infty} \alpha_{n} (1-
\alpha_{n})$ diverges, $\sum_{n=1}^{\infty} \alpha_{n} (1-
\beta_{n}) $ converges and $\limsup \beta_{n} < 1$, then Ishikawa
iteration process converges weakly to a fixed point of $T$.
\par
It easy to see that process (\ref{ishikawa}) is more general than
the process (\ref{mann}). Also, for a Lipschitz pseudocontractive
mapping in a Hilbert space, process (\ref{mann}) is not known to
converge to a fixed point while the process (\ref{ishikawa}) is
convergent. In spite of these facts, researchers are interested to
study the convergence theorems by process (\ref{mann}), because of
the formulation of process (\ref{mann}) is simpler than that of
(\ref{ishikawa}) and if $\{\beta_{n}\}$ satisfies suitable
conditions, we can gain a convergence theorem for process
(\ref{ishikawa}) on a convergence theorem for process (\ref{mann}).
\par
In recent years, many authors have proved weak or strong convergence
theorems for some nonlinear mappings by using various iteration
processes in the framework of Hilbert spaces and Banach spaces, see,
\cite{Moradlou, Nakajo, taka2003, T5}.
\par
Let $C$ be a nonempty, closed convex subset of a real Banach
space $E$. A mapping $S$ from $C$ into $E$ is called
\textit{symmetric generalized hybrid} \cite{T4} if there exist
$\alpha, \beta ,\gamma,\delta\in \mathbb{R}$ such that
 \begin{equation}\begin{aligned}\label{inq.}
\alpha\|Sx-Sy\|^2+\beta(\|x-Sy\|^2&+\|Sx-y\|^2)+\gamma\|x-y\|^2\\
&+\delta(\|x-Tx\|^2+\|y-Sy\|^2)\leq 0,
\end{aligned}\end{equation}
for all $x,y\in C$. We call such a mapping an
$(\alpha,\beta,\gamma,\delta)$-symmetric generalized hybrid mapping.
\par
In this paper, motivated by Takahashi and Yao \cite{T5}, we  prove
some theorems related to properties of generalized symmetric hybrid
mappings in Banach spaces. Moreover,  we prove a fixed point theorem
for symmetric generalized hybrid mappings in a Banach space . Also,
we prove some weak convergence theorems for symmetric generalized
hybrid mappings in a uniformly convex Banach space.
\section{Preliminaries}
  Let $E$ be a real Banach space with $\|.\|$ and dual space $E^{*}$.
  We denote by $J$ the normalized duality mapping from $E$ into $2^{E^{*}}$ defined by
 $$Jx=\{x^{*}\in E^{*}:\; \langle x,x^{*}\rangle=\|x\|^2=\|x^{*}\|^2\},$$
 for all $x\in E$, where $\langle.,.\rangle$ denotes the generalized duality pairing between $E$ and $E^{*}$.
  $E$ is said to be strictly convex if  $\|\frac{x+y}{2}\|<1$ for all
  $x,y\in E$ with $\|x\|=\|y\|=1$ and $x\neq y.$ It is also said to be uniformly convex if for every $\epsilon\in (0,2]$,
  there exists a $\delta>0$, such that $\|\frac{x+y}{2}\|<1-\delta$ for all $x,y\in E$ with $\|x\|=\|y\|=1$ and
  $\|x-y\|\geq\epsilon$. Furthermore, $E$ is called smooth if the limit
 \begin{equation}\label{eq1}
 \lim_{t\rightarrow0}\frac{\|x+ty\|-\|x\|}{t},
 \end{equation}
 exists for all $x,y\in B_{E}=\{x\in E:\;\|x\|=1\}$. It is also said to be uniformly smooth if the limit
  $(\ref{eq1})$ is attained uniformly for all $x,y\in E$. For more details see \cite{Agarwal, T3}.
\par
Denote by $l^{\infty}$
the set of all bounded sequences equipped with supremum norm. A
continuous linear functional $\mu$ on $l^{\infty}$ is called {\it
Banach limit} if
\begin{itemize}
  \item[i)] $\mu(e)=\|e\| = 1$, where $e=(1,1,1,\cdots)$.
  \item[ii)] $\mu_{n}(x_{n}) = \mu_{n}(x_{n+1})$ for all $x = (x_1, x_2, \cdots ) \in
  l^{\infty}$, where \linebreak$ \mu_{n}(x_{n+m})= \mu(x_{m+1}, x_{m+2},x_{m+3}, \cdots, x_{m+n},
  \cdots)$.
\end{itemize}
As usual, we denote by $\mu_{n}(x_{n})$ the value of $\mu$ at $x =
(x_1, x_2, \cdots )$.  It is well known that there exists a Banach
limit on $l^{\infty}$. Let $\mu$ be a Banach limit, then
\begin{equation*}
\liminf_{n \to \infty}x_{n}\leq \mu_{n}(x_{n})\leq \limsup_{n\to
\infty} x_{n}.
\end{equation*}
Moreover, if $x_{n} \to a$, then $ \mu_{n}(x_{n})=a$. For more
details we refer readers to \cite{Agarwal, T3}.
\par
 We denote the weak convergence and the strong convergence of $\{x_n\}$ to $x\in E$ by $x_n\rightharpoonup x$
 and $x_n\rightarrow x$, respectively.
\par
A Banach space $E$  satisfies the Opial's condition if for every
sequence $\{x_n\}$ in $E$ such that $x_n\rightharpoonup x\in E$,
then
$$\liminf_{n\rightarrow\infty}\|x_n-x\|<\liminf_{n\rightarrow\infty}\|x_n-y\|,$$
for all $y\in E$, $y\neq x$.
\par
 A self mapping $T$ of $C\subseteq E$ is called:
(i) \textit{firmly nonexpansive} \cite{Bruck}, if
$\|Tx-Ty\|^{2}\leq \langle x-y,j\rangle$ for all $x,y \in C$, where $j\in J(Tx-Ty)$;
(ii)\textit{ nonspreading}, if $2\|Tx-Ty\|^{2}\leq \|Tx-y\|^{2} +
\|Ty-x\|^{2}$ for all $x,y \in E$; (iii) \textit{hybrid}, if
$3\|Tx-Ty\|^{2}\leq \|x-y\|^{2}+\|Tx-y\|^{2} + \|Ty-x\|^{2}$ for all
$x,y \in E$. Also, a  self mapping $T$ of $C$ with $F(T)\neq \emptyset$ is called \textit{quasi-nonexpansive}
if $\|x - Ty\| \leq \| x - y\|$ for all $x \in F(T)$ and $y \in C$.
\par
 It easy to see that:
\begin{itemize}
  \item $(1,0,-1,0)$-symmetric generalized hybrid mapping is
nonexpansive mapping;
  \item $(2,-1,0,0)$-symmetric generalized hybrid mapping is
nonspreading mapping;
  \item $(3,-1,-1,0)$-symmetric generalized hybrid mapping is hybrid
mapping.
\end{itemize}
  The following result is given in \cite{Agarwal, T3}.
\begin{thm}\label{thj}
Let $E$ be a Banach space and let $J$ be the duality mapping of $E$.
Then
$$\|x\|^2-\|y\|^2\geq2\langle x-y,j\rangle,$$
for all $x,y\in E$ where $j\in J_y.$
\end{thm}
 \begin{thm}\cite{Hsu}\label{th0}
 Let $C$ be a nonempty closed convex subset of a uniformly convex
Banach space $E$ and $T$ be a self mapping of $C$. Let $\{x_n\}$ be
a bounded sequence of $E$ and $\mu$ be a mean on $l^\infty$. If
$$\mu_n\|x_n-Tu\|^2\leq\mu_n\|x_n-u\|^2,$$
for all $u\in C$, then $T$ has a fixed point in $C$.
\end{thm}
  \begin{thm}\cite{Xu}\label{thm1}
Let $E$ be a uniformly convex Banach space and let $r$ be a positive
real number. Then there exists a strictly increasing, continuous and
convex function $g :[0,\infty)\longrightarrow [0,\infty)$ such that
$g(0) = 0$ and
$$\|t x+(1-t)y\|^2\leq t\|x\|^2+(1-t)\|y\|^2-t(1-t)g(\|x-y\|)$$
for all $x, y \in B_r$ and $t$ with $0 \leq t \leq1$, where $B_r
= \{z\in E : \|z\| \leq r\}$.
\end{thm}
\section{Main results}
  \begin{thm}\label{thm3.1}
Let $E$ be real Banach space, $C$ be a nonempty closed convex subset
of $E$ and $T$ be an $(\alpha,\beta,\gamma,\delta)$-symmetric
generalized hybrid self mapping of $C$ such that $F(T)\neq\phi$  and
the conditions (1) $\alpha+2\beta+\gamma\geq0$, (2)
$\alpha+\beta>0 $ and (3) $\delta\geq0$ hold. Then $T$ is quasi-nonexpansive.
\end{thm}
\begin{proof}
Since $T$ is an $(\alpha,\beta,\gamma,\delta)$-symmetric generalized
hybrid self mapping of $C$, then
 \begin{equation}\begin{aligned}\label{thm1.1}
\alpha\|Tx-Ty\|^2+\beta(\|x-Ty\|^2&+\|Tx-y\|^2)+\gamma\|x-y\|^2\\
&+\delta(\|x-Tx\|^2+\|y-Ty\|^2)\leq0,
\end{aligned}\end{equation}
for all $x,y\in E$. Since $F(T)\neq\phi$, hence there exists $x\in
E$ such that $x=Tx$. So,
 \begin{equation*}
\alpha\|x-Ty\|^2+\beta(\|x-Ty\|^2+\|x-y\|^2)+\gamma\|x-y\|^2+\delta\|y-Ty\|^2\leq0,
\end{equation*}
for all $y\in E$. Therefore we can conclude that
 \begin{equation}\label{inq1.1}
(\alpha+\beta)\|x-Ty\|^2+(\beta+\gamma)\|x-y\|^2\leq0,
\end{equation}
for all $y\in E$.  It follows from condition (2) and (\ref{inq1.1})
that $-(\beta+\gamma)\geq0$. So conditions (1) and (2) imply that
 \begin{equation}\label{inq1.2}
 0\leq\frac{-(\beta+\gamma)}{\alpha+\beta}\leq1.
 \end{equation}
Then, form (\ref{inq1.1})  and (\ref{inq1.2}), we derive that  $\|x-Ty\|\leq\|x-y\|$, i.e., $T$
is quasi-nonexpansive. .
\end{proof}
\begin{thm}\label{thm3.2}
  Let $E$ be a real Banach space, $C$ be a nonempty subset of $E$ and $\zeta,\eta$ be nonnegetive real numbers.
  Then a firmly nonexpansive self  mapping  of $C$ is a $(2\zeta+\eta,-\zeta,-\eta,0)$-symmetric generalized hybrid mapping.
  \end{thm}
  \begin{proof}
  Assume that  $T$ is a firmly nonexpasive self mapping of $C$. Then we have
  $$\|Tx-Ty\|^2\leq\langle x-y,j\rangle,$$
  for all $x,y\in C$ and $j\in J(Tx-Ty).$ By using Theorem \ref{thj}  we get
  \begin{equation*}\begin{aligned}
  \|Tx-Ty\|^2\leq\langle x-y,j\rangle &\Longleftrightarrow0\leq2\langle x-Tx-(y-Ty),j\rangle\\
  &\Longrightarrow0\leq\|x-y\|^2-\|Tx-Ty\|^2\\
  &\Longleftrightarrow\|Tx-Ty\|^2\leq\|x-y\|^2\\
  &\Longleftrightarrow\|Tx-Ty\|\leq\|x-y\|.
  \end{aligned}
  \end{equation*}
  Hence for $\zeta\geq0$, we have
  \begin{equation}\label{eq1.j}
  \zeta\|Tx-Ty\|\leq\zeta\|x-y\|.
  \end{equation}
  On the other hand, for all $x,y\in C$ and $j\in J(Tx-Ty)$ we get
 \begin{equation*}\begin{aligned}
  \|Tx-Ty\|^2\leq\langle x-y,j\rangle&\Longleftrightarrow0\leq2\langle x-Tx-(y-Ty),j\rangle\\
   &\Longleftrightarrow0\leq2\langle x-Tx,j\rangle+2\langle Ty-y,j\rangle\\
  &\Longrightarrow0\leq\|x-Ty\|^2-\|Tx-Ty\|^2+\|Tx-y\|^2-\|Tx-Ty\|^2\\
   &\Longleftrightarrow0\leq\|x-Ty\|^2-\|y-Tx\|^2-2\|Tx-Ty\|^2\\
  &\Longleftrightarrow2\|Tx-Ty\|^2\leq\|x-Ty\|+\|y-Tx\|^2.
  \end{aligned}
  \end{equation*}
  So, for $\eta\geq0$ we get
  \begin{equation}\label{eq2.j}
 2\eta\|Tx-Ty\|^2\leq\eta\|x-Ty\|+\eta\|y-Tx\|^2.
  \end{equation}
  Hence, by summing both sides of $(\ref{eq1.j})$ and $(\ref{eq2.j})$  we obtain
  $$(\zeta+2\eta)\|Tx-Ty\|^2\leq\eta\|x-Ty\|^2+\eta\|y-Tx\|^2+\zeta\|x-y\|^2.$$
  and therefore
  $$(\zeta+2\eta)\|Tx-Ty\|^2-\eta(\|x-Ty\|^2+\|y-Tx\|^2)-\zeta\|x-y\|^2\leq0.$$
  This yields that $T$ is a $(\zeta+2\eta,-\eta,-\zeta,0)$-symmetric generalized hybrid mapping.
  \end{proof}
\begin{thm}\label{thm3.3}
 Let $C$ be a nonempty closed convex subset of  a real Banach space $E$ and $T$ be an
$(\alpha,\beta,\gamma,\delta)$-symmetric generalized hybrid self
mapping of $C$ and the conditions (1) $\alpha+2\beta+\gamma\geq0$, (2)
$\alpha+\beta>0 $ and (3) $\delta\geq0$ hold. Then the following are equivalent:
\begin{itemize}
\item[(i)]$F(T)\neq\phi$;
\item[(ii)]$\{T^nx\}$ is bounded for some $x\in C$.
\end{itemize}
\end{thm}
\begin{proof}
$(i)\Longrightarrow (ii)$: It is obvious. \newline
 $(ii)\Longrightarrow (i)$:
Since $T$ is an $(\alpha,\beta,\gamma,\delta)$-symmetric generalized
hybrid self mapping of $C$, then the inequality \eqref{thm1.1} is
satisfied. Let $u\in C$ such that $\{T^nu\}$ is bounded. Replacing
$x$ by $T^{n} u$  in \eqref{thm1.1}, we have
\begin{equation*}\begin{aligned}
\alpha\|T^{n+1}u-Ty\|^2&+\beta\|T^nu-Ty\|^2\\
&\leq-\beta\|T^{n+1}u-y\|^2-\gamma\|T^nu-y\|^2 -\delta(\|T^{n}
u-T^{n+1} u\|^2+\|y-Ty\|^2)\\
&\leq -\beta\|T^{n+1}u-y\|^2-\gamma\|T^nu-y\|^2,
\end{aligned}\end{equation*}
for all $y\in C$ and $n\in\mathbb{N}$. Since $\{T^nu\}$ is bounded,
by taking a Banach limit $\mu$ on both sides the last inequlity, we
get
\begin{equation*}\begin{aligned}
\mu_n(\alpha\|T^{n+1}u-Ty\|^2&+\beta\|T^nu-Ty\|^2)\\
&\leq\mu_n(-\beta\|T^{n+1}u-y\|^2-\gamma\|T^nu-y\|^2).
\end{aligned}\end{equation*}
So, by using the properties of Banach limit, we have
\begin{equation*}\begin{aligned}
\alpha\mu_n\|T^{n}u-Ty\|^2&+\beta\mu_n\|T^nu-Ty\|^2\\
&\leq-\beta\mu_n\|T^{n}u-y\|^2-\gamma\mu_n\|T^nu-y\|^2.
\end{aligned}\end{equation*}
From the last inequality, we can conclude that
\begin{equation*}
(\alpha+\beta)\mu_n\|T^{n}u-Ty\|^2\leq-(\beta+\gamma)\mu_n\|T^nu-y\|^2,
\end{equation*}
Similar to the proof of Theorem \ref{thm3.1}, we derive that
\begin{equation*}
\mu_n\|T^{n}u-Ty\|^2\leq\mu_n\|T^nu-y\|^2,
\end{equation*}
for all $y\in C$. So Theorem \ref{th0} implies that $T$ has a fixed
point.

\end{proof}

  \begin{thm}\label{3.4}
  Let $C$ be a nonempty closed convex subset of a Banach space $E$ satisfying Opial's condition.
  Assume that  $T$ be an $(\alpha,\beta,\gamma,\delta)$-symmetric generalized hybrid self mapping
  of  $C$ such that the conditions (1) $\alpha+2\beta+\gamma\geq0$, (2) $\alpha+\beta>0 $, (3) $\beta \leq 0$
    and (4) $\delta\geq 0$ hold. Then $I-T$ is demiclosed (at $0$), i.e., $x_n\rightharpoonup u$ and
    $x_n-Tx_n\rightarrow0$ imply $u\in F(T).$
\begin{proof}
Since $T$ is an $(\alpha,\beta,\gamma,\delta)$-symmetric generalized
hybrid  self mapping of $C$, so the inequality \eqref{thm1.1} is
satisfied. Assume that $x_n\rightharpoonup u$ and
$x_n-Tx_n\rightarrow0$. Since $x_n\rightharpoonup u$, we can conclude that $\{x_n\}$ is
bounded and by $\lim_{n\to\infty}\|x_n-Tx_n\|=0$ we obtain that $\{Tx_n\}$ is
bounded. Substituting $x$ and $y$ by $x_n$ and $u$ in
\eqref{thm1.1}, respectively, we have
\begin{equation*}
 \begin{aligned}
\alpha\|Tx_n-Tu\|^2+\beta(\|x_n-Tu\|^2&+\|Tx_n-u\|^2)+\gamma\|x_n-u\|^2\\
&+\delta(\|x_n-Tx_n\|^2+\|u-Tu\|^2)\leq0.
\end{aligned}
\end{equation*}
Therefore
\begin{equation*}
\alpha\|Tx_n-Tu\|^2\leq-\beta\|x_n-Tu\|^2-\beta\|Tx_n-u\|^2-\gamma\|x_n-u\|^2,
\end{equation*}
and hence
\begin{equation*}
 \begin{aligned}
\alpha\|Tx_n-Tu\|^2\leq-\beta(\|Tx_n-x_n\|&+\|x_n-u\|)^2-\gamma\|x_n-u\|^2\\
&-\beta(\|x_n-Tx_n\|+\|Tx_n-Tu\|)^2.
\end{aligned}
\end{equation*}
So, we can conclude that get
\begin{equation}\label{inq}
 \begin{aligned}
\|Tx_n-Tu\|^2&\leq\frac{-(\beta+\gamma)}{\alpha+\beta}\|x_n-u\|^2-\frac{2\beta}{\alpha+\beta}\|x_n-Tx_n\|^2\\
&\hspace{.5cm}-\frac{2\beta}{\alpha+\beta}(\|x_n-u\|+\|Tx_n-Tu\|)\|Tx_n-x_n\|\\
&\leq\|x_n-u\|^2-\frac{2\beta}{\alpha+\beta}\|x_n-Tx_n\|^2\\
&\hspace{.5cm}-\frac{2\beta}{\alpha+\beta}(\|x_n-u\|+\|Tx_n-Tu\|)\|Tx_n-x_n\|.
\end{aligned}
\end{equation}
Assume that $Tu\neq u$. So by using boundedness of $\{x_n\}$ and $\{Tx_n\}$, Opial's condition and (\ref{inq}), we have
\begin{equation}
\begin{aligned}
\liminf_{n\to\infty}\|x_n-u\|^2&<\liminf_{n\to\infty}\|x_n-Tu\|^2\\
&=\liminf_{n\to\infty}\|Tx_n-Tu\|^2\\
&\leq\liminf_{n\to\infty}(\leq\|x_n-u\|^2-\frac{2\beta}{\alpha+\beta}\|x_n-Tx_n\|^2\\
&\hspace{2cm}-\frac{2\beta}{\alpha+\beta}(\|x_n-u\|+\|Tx_n-Tu\|)\|Tx_n-x_n\|)\\
&\leq\liminf_{n\to\infty}\|x_n-u\|^2,
\end{aligned}
\end{equation}
which is a contradiction. Hence we get $Tu=u$ and therefore $I-T$ is
demiclosed.
\end{proof}
\end{thm}
\begin{thm}\label{3.5}
Let $C$ be a nonempty closed convex subset of a uniformly convex Banach space $E$ satisfying Opial's condition.
 Assume that  $T$ is an $(\alpha,\beta,\gamma,\delta)$-symmetric generalized hybrid self mapping of $C$ such that
  $F(T)\neq\phi$ and the conditions (1)  $\alpha+2\beta+\gamma\geq0$, (2) $\alpha+\beta>0 $, (3) $\beta \leq 0$
    and (4) $\delta\geq 0$ hold.  Assume that $\{x_n\}$ is a sequence generated
by
\begin{equation*}
\begin{cases}
&x_1=x\in C,\\
&y_n=(1-\lambda_n)x_n+\lambda_nTx_n,\\
&x_{n+1}=(1-\gamma_n)x_n+\gamma_nTy_n,
\end{cases}
\end{equation*}
where $0\leq\lambda_n\leq1$, $0<a\leq\gamma_n\leq1$ for all $n\in\mathbb{N}$ and
$\liminf_{n\to\infty}\lambda_n(1-\lambda_n)>0$. Then $x_n\rightharpoonup x_0\in F(T)$.
\end{thm}
\begin{proof}
Since $T$ is an $(\alpha,\beta,\gamma,\delta)$-symmetric generalized  hybrid mapping such that $F(T)\neq\phi$, so by Theorem $\ref{thm3.1}$, $T$ is quasi-nonexpansive. Then, for all $q\in
F(T)$ and all $n \in \mathbb{N}$,  we
have\begin{equation}\label{lm1.1}\begin{aligned}
    \|y_n-q\|
    &=\|(1-\lambda_n)x_n+\lambda_nTx_n-q\|\\
    &=\|(1-\lambda_n)(x_n-q)+\lambda_n\|Tx_n-q\|\\
    &\leq(1-\lambda_n)\|x_n-q\|-\lambda_n\|x_n-q\|\\
    &=\|x_n-q\|,
\end{aligned}
\end{equation}
and hence by using $(\ref{lm1.1})$, we get
\begin{equation}\label{inq1}
\begin{aligned}
    \|x_{n+1}-q\|
    &=\|(1-\gamma_n)x_n+\gamma_nTy_n-q\|\\
    &\leq(1-\gamma_n)\|x_n-q\|+\gamma_n\|Ty_n-q\|\\
    &\leq(1-\gamma_n)\|x_n-q\|+\gamma_n\|y_n-q\|\\
    &\leq\|x_n-q\|.
\end{aligned}
\end{equation}
Then, we can conclude that $\lim_{n\to\infty}\|x_n-q\|$ exists. So, $\{x_n\}$ and $\{y_n\}$ are bounded. Since $T$ is quasi-nonexpansive, $\{Tx_n\}$ and $\{Ty_n\}$ are also bounded. Let
$$ r=\max\{\sup_{n\in\mathbb{N}}\|x_n-q\|,\;\sup_{n\in\mathbb{N}}\|Tx_n-q\|,\;\sup_{n\in\mathbb{N}}\|y_n-q\|,\;\sup_{n\in\mathbb{N}}\|Ty_n-q\|\}.$$
Hence, by Theorem \ref{thm1}, there exists a strictly increasing, continuous and convex function $g :[0,\infty)\longrightarrow
[0,\infty)$ such that $g(0) = 0$ and
$$\|t x+(1-t)y\|^2\leq t\|x\|^2+(1-t)\|y\|^2-t(1-t)g(\|x-y\|),$$
for all $x, y \in B_r$
and $t$ with $0 \leq t \leq1$, where $B_r = \{z\in E : \|z\| \leq r\}$. Then, for all $q\in F(T), x\in C$ and $n\in\mathbb{N}$, we get
\begin{equation}\begin{aligned}
    \|y_n-q\|^2
    &=\|(1-\lambda_n)x_n+\lambda_nTx_n-q\|^2\\
    &=\|(1-\lambda_n)(x_n-q)+\lambda_n(Tx_n-q)\|^2\\
    &\leq(1-\lambda_n)\|x_n-q\|^2+\lambda_n\|x_n-q\|^2-\lambda_n(1-\lambda_n)g(\|x_n-Tx_n\|)\\
    &=\|x_n-q\|^2-\lambda_n(1-\lambda_n)g(\|x_n-Tx_n\|),
\end{aligned}
\end{equation}
and hence
\begin{equation}\label{inq1}
\begin{aligned}
    \|x_{n+1}-q\|^2
    &=\|(1-\gamma_n)x_n+\gamma_nTy_n-q\|^2\\
    &=\|(1-\gamma_n)(x_n-q)+\gamma_n(Ty_n-q)\|^2\\
    &\leq(1-\gamma_n)\|x_n-q\|^2+\gamma_n\|y_n-q\|^2-\gamma_n(1-\gamma_n)g(\|x_n-Ty_n\|)\\
    &\leq(1-\gamma_n)\|x_n-q\|^2+\gamma_n\|x_n-q\|^2-\gamma_n\lambda_n(1-\lambda_n)g(\|x_n-Tx_n\|)\\
    &-\gamma_n(1-\gamma_n)g(\|x_n-Ty_n\|)\\
    &\leq\|x_n-q\|^2-\gamma_n\lambda_n(1-\lambda_n)g(\|x_n-Tx_n\|).
\end{aligned}
\end{equation}
Since $0<a\leq\gamma_n\leq1,$
it is easy to see that
$$\|x_{n+1}-q\|^2\leq\|x_n-q\|^2-a\lambda_n(1-\lambda_n)g(\|x_n-Tx_n\|).$$
So,
$$ 0\leq a \lambda_n(1-\lambda_n)g(\|x_n-Tx_n\|)\leq\|x_n-q\|^2-\|x_{n+1}-q\|^2\rightarrow0,$$
as $n\rightarrow\infty$, since
$\liminf_{n\to\infty}\lambda_n(1-\lambda_n)>0$. Therefore
\begin{equation*}
\lim_{n\to\infty}g(\|x_n-Tx_n\|)=0.
\end{equation*}
From the properties of $g$, we get
\begin{equation}\label{eq2}
\lim_{n\to\infty}\|x_n-Tx_n\|=0.
\end{equation}
\par
Now, we conclude from boundedness of $\{x_n\}$  and reflexivity of $E$ that  there exists a subsequence $\{x_{n_i}\}$
of $\{x_n\}$ such that $x_{n_i}\rightharpoonup q\in C$. So Theorem \ref{3.4} and (\ref{eq2}) imply that
$Tq = q$. We will prove that the sequence $\{x_n\}$ converges weakly to some point
of $F(T)$. Suppose that  there exist  two subsequences $\{x_{n_i}\}$ and $\{x_{n_j}\}$
of $\{x_n\}$ such that $x_{n_i}\rightharpoonup q$ and $x_{n_i}\rightharpoonup p$. Assume that $q\neq p$. We know that $\lim_{n\to\infty}\|x_n-q\|$ and $\lim_{n\to\infty}\|x_n-p\|$ exist, since $q, p\in F (T)$. So, Opial's condition on  $E$ implies that
\begin{equation*}
\begin{aligned}
\lim_{n\to\infty}\|x_n-q\|=\lim_{i\to\infty}\|x_{n_i}-q\|<\lim_{i\to\infty}\|x_{n_i}-p\|=\lim_{n\to\infty}\|x_{n}-p\|
&=\lim_{j\to\infty}\|x_{n_j}-p\|\\
&<\lim_{j\to\infty}\|x_{n_j}-q\|\\&=\lim_{n\to\infty}\|x_{n}-q\|.
\end{aligned}
\end{equation*}
Which is a contradiction. Therefore, we obtain $q = p$. This yields
that $\{x_n\}$
 converges weakly to a point of $F(T)$.
\end{proof}

  \begin{thm}
 Let $C$ be a nonempty closed convex subset of a uniformly convex Banach space $E$ satisfying Opial's condition. Suppose that  $T$ is a hybrid self mapping of $C$ with $F(T)\neq\phi$.  Assume that $\{x_n\}$ is a sequence generated by
\begin{equation*}
\begin{cases}
&x_1=x\in C,\\
&y_n=(1-\lambda_n)x_n+\lambda_nTx_n,\\
&x_{n+1}=(1-\gamma_n)x_n+\gamma_nTy_n,
\end{cases}
\end{equation*}
where $0\leq\lambda_n\leq1$, $0 < a\leq\gamma_n\leq1$ for all $n\in\mathbb{N}$ and
$\liminf_{n\to\infty}\lambda_n(1-\lambda_n)>0$. Then $x_n\rightharpoonup x_0\in F(T)$.
\end{thm}
\begin{proof}
Since $T$ is a hybrid self mapping of $C$, so $T$ is a $(3,-1,-1,0)$-symmetric generalized hybrid mapping. Therefore by Theorem \ref{3.5}, we get the desired result.
\end{proof}
\begin{rk}
Since nonexpansive mappings are a $(1,0,-1,0)$-symmetric generalized
hybrid mapping and nonspreading mappings are a $(2,-1,0,0)$-symmetric
generalized hybrid mapping, then the Theorem \ref{3.5} holds for
these mappings.
\end{rk}

\end{document}